\documentclass{article}

\usepackage{url,amssymb,amsmath,amsthm,graphicx,xcolor,dsfont,parskip,calc,tabularx,pdflscape,arydshln}

\newcommand{\bmat}[1]{\begingroup%
	\begin{bmatrix} #1 \end{bmatrix}
	\endgroup}

\newcommand{\ind}{\scalebox{1.15}{$\mathds{1}$}}
\newcommand{\one}{\scalebox{1.1}{$\mathbf{1}$}}
\newcommand{\reals}{\ensuremath{\mathbb{R}}}
\newcommand{\integers}{\ensuremath{\mathbb{Z}}}
\newcommand{\coloneq}{\stackrel{\textup{\tiny def}}{=}}
\newcommand{\eqcolon}{\stackrel{\textup{\tiny def}}{=}}

\newtheorem{theorem}{Theorem}
\newtheorem{remark}{Remark}

\newcommand{\J}{\mathcal{K}}
\DeclareMathOperator{\prob}{prob}
\DeclareMathOperator{\trace}{tr}
\DeclareMathOperator{\card}{card}

\newcommand{\abs}[1]{\left\lvert  #1 \right\rvert}
\newcommand{\norm}[1]{\left\lVert  #1 \right\rVert}

\newcommand{\proj}{\ensuremath{\textrm{proj}}}

\newcommand{\skipthis}[1]{{}}

\title{Resource Allocation with Population Dynamics}

\author{Jonathan Epperlein and Jakub Mare{\v c}ek\thanks{J. Epperlein and J. Mare{\v c}ek are at 
              IBM Research -- Ireland, B3 IBM Campus Damastown, Dublin 15, Ireland. 
}
}

\date{}

\usepackage[%
breaklinks=false,colorlinks=true,citecolor=black,urlcolor=blue,linkcolor=black,pdfpagemode=none,pdfstartview=FitV%
]{hyperref}

\begin{document}

	\maketitle

\begin{abstract}
Many analyses of resource-allocation problems employ simplistic models of the population.
Using the example of a resource-allocation problem of
Mare{\v c}ek et al.\ [Int. J. Control 88(10), 2015], 
we introduce rather a general behavioural model, where the evolution of a heterogeneous population of agents is governed by a Markov chain.
Still, we are able to show that the distribution of agents across resources converges in distribution, 
for suitable means of information provision, under certain assumptions.
The model and proof techniques may have wider applicability.
\end{abstract}

%
\section{Introduction}\label{sec:intro}

 There are resource-allocation problems encountered in almost every aspect of human lives: 
 from utilities such as power systems and water systems, to transportation, and office space allocation.
 Many analyses of resource-allocation problems employ simplistic models of the population
 which ignore much of the complexity of human behaviour.
 Notice, for example, that the demand for a resource is often non-stationary, as exemplified by the work-day morning rush hour in transportation
 and the existence of predictable peaks in the demand in many other domains.
 Notice, further, that humans may have access to only very limited amount of information, 
 but may still consider multiple criteria, and that their appreciation of the criteria may vary over time.
 As an example of a particular resource-allocation problem, we introduce a model of behaviour and the related demand process, 
 which captures both the multi-criteria aspects of the decision making and non-stationarity of the demand process. 
 Still, we show that the distribution of agents across resources converges in distribution, 
 for suitable means of information provision and under certain assumptions.

 As our running example, we consider the problems faced by transportation authorities in charge of a road network composed of a number of road segments.
 For each road segment, the travel time is, in principle, a time series with a data point per a vehicle passing across the road segment.
 Transportation authorities increasingly have information on traffic conditions throughout the road network at this level of detail, 
 but want to broadcast much less information to the public.
 Mare\v{c}ek et al.~\cite{marevcek2015signaling,MarecekShortenYu2016b,marevcek2016signaling} have recently demonstrated that such information provision should be seen as means
 of modulating the demand process and that different means of modulation may lead to very 
 different outcomes. 
 Whereas, for example, it is natural for each user to forecast the travel time as a scalar for each route, if the transportation authorities announce
 two distinct scalar values for two distinct alternative routes, the route with the lower scalar value announced may become congested, and if the authorities are truthful in reporting the congestion, 
 the congestion may alternate between the two alternative routes in a sub-optimal limit-cycle behaviour, \emph{ad infinitum}. 
 In contrast, providing two intervals for two distinct alternative routes may benefit
 from the uncertainty in the response of the population.
 Mare\v{c}ek et al. used a rather simple behavioural model, where the response of each agent
 is determined by its level of risk aversion \cite{marevcek2016signaling,marevcek2015signaling} and possibly actuation delay \cite{marevcek2015signaling}, and the distribution of levels of risk aversion
 in the population is sampled in a memory-less fashion.

 In Section~\ref{sec:mod}, we introduce a considerably more realistic behavioural model, where the evolution of a heterogeneous population of agents is governed by a Markov chain. Section~\ref{sec:ifs} gives a brief summary of relevant results in the theory of iterated function systems and their extension to recurrent iterated function systems; using these results, we show in Section~\ref{sec:main} that the distribution of agents across resources converges in distribution, under certain assumptions. Section~\ref{sec:computational} illustrates the model and main result with a few simulations, and we explain the importance of the result in the context of related work in Section~\ref{sec:related}, before concluding with a variety of suggestions for future work in the field. Appendix~\ref{sec:app:P} details the crucial step of generating a transition matrix for the Markov chain governing the population dynamics. 


\section{The Model}\label{sec:mod}

	Let us consider a model consisting of three components \cite{marevcek2016signaling,marevcek2015signaling}: a set of resources, a central authority that provides information about the resources,
   and a population of agents, who make decisions as to which resource to use, based solely on the information made available by the central authority and their personal cost functions. 
    Further, let us extend this model with population dynamics in Section \ref{ssec:popdyn}. 
    A summary of the model is presented in Figure~\ref{fig:modflow}.
	In the following, we describe each of the components in detail. 

\begin{figure}
	\centering
	\includegraphics[width=0.9\textwidth]{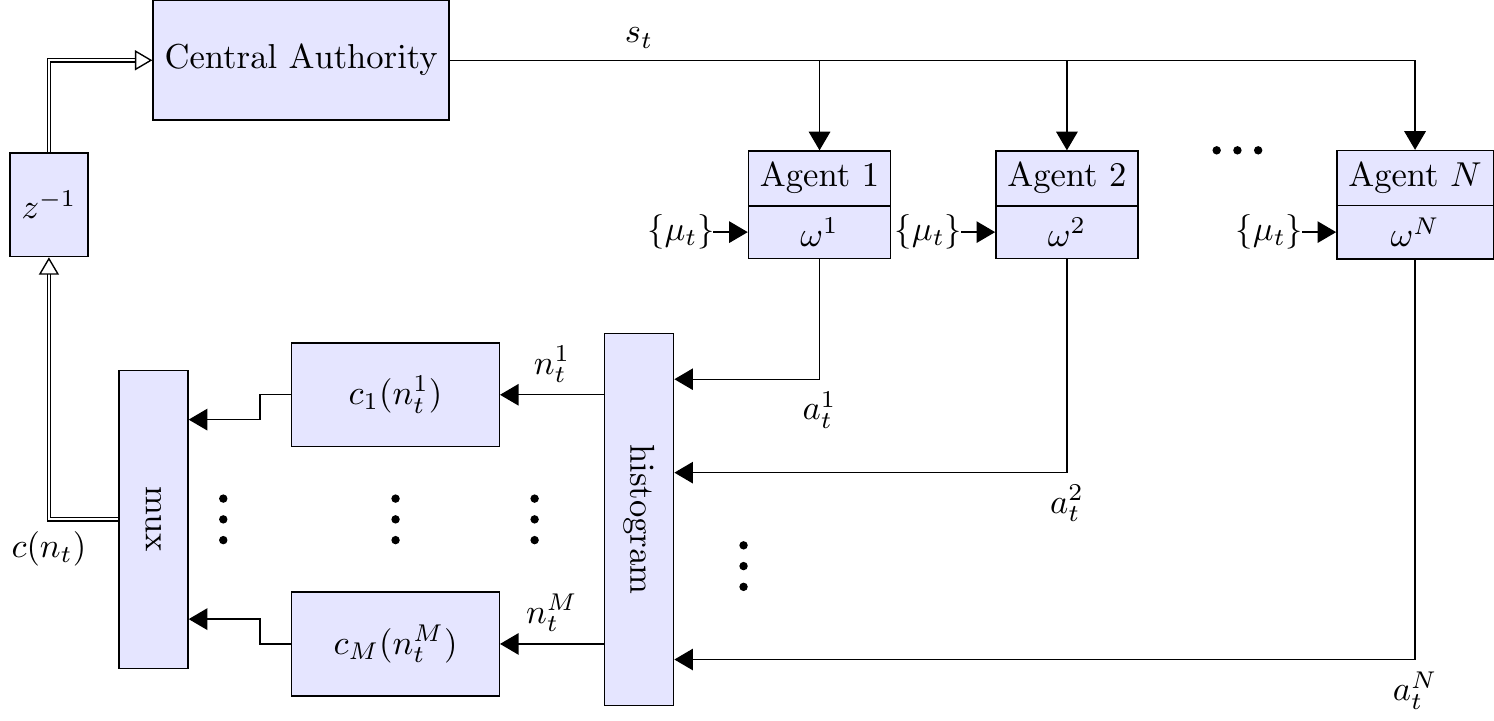}
	\caption{A signal flow of the model described in Section~\ref{sec:mod}: The central authority broadcasts at time $t$ the signal $s_t$ based on information collected up until then (which does not include $t$ itself). Every agent $i$ makes its decision $a_t^i$ which resource to use based on its policy $\omega^i$. The decisions of all agents generate the new congestion profile $n_t$, which then leads to the current cost profile $c(n_t)$; $z^{-1}$ denotes a unit delay.}
	\label{fig:modflow}
\end{figure}

\subsection{Resources}
There is a finite set $\{1,2,\dotsc,M\}$ of resources, from which each agent chooses exactly one at every time step. For example in transportation, the resources would correspond to routes from origins to destinations provided by a road infrastructure. The cost of using such a resource is modelled by continuous cost functions (``link performance functions'') of the form
$
		c_m(n_t^m),
$
where the index $m$ denotes the resource, and $n_t^m$ denotes the number of agents choosing resource $m$ at time step $t$. Implicit in this form is the assumption of separability: the performance of resource $m$ depends only on the utilization of resource $m$ itself. 
In the transportation applications, the cost would correspond to the travel time of a single driver along the route. 
The vector ${n}_t \coloneq\bmat{n_t^1&\dotsm&n_t^M}$ is the \textit{congestion profile}. 
\subsection{Central Authority and Signalling Scheme}
The central authority has access to the history of congestion profiles $\{{n}_\tau\}_{\tau=1}^{t-1}$ and provides information derived from this history to all agents, with the aim of lowering the social cost, $\sum_{m = 1}^{M} (n_t^m / N) c_m(n_t^m)$, while being truthful, and hence trustworthy.   
Here, the same \textit{two} pieces of information, $u_t^m$ and $v_t^m$, are provided to all agents about \textit{each} resource $m$. 
One can imagine $u_t^m$ and $v_t^m$ are, for example, the minimum and maximum costs within a time window, mean and variance over the same time window, or information about travel time and tolls at resource $m$; see Table~\ref{tab:int-signals} for details and references. 
In all cases listed, the signal $s_t = \bmat{u_t^1 & v_t^1 & u_t^2 & \dotsm &u_t^M & v_t^M}$ remains in $\reals^{2M}$. 
	
	For this paper, we use exponential smoothing on the past costs of resource $m$ to obtain $u_t^m$, whereas $v_t^m$ provides a measure of their volatility:
	\begin{equation}\label{eq:uv}
	\begin{split}
		u_t^m & \coloneq (1-q_1)\Bigl( c_m(n_{t-1}^m) + q_1 c_m(n_{t-2}^m) + \dotsb + q_1^{t-1} c_m(n_{1}^m) \Bigr)\\
		v_t^m & \coloneq (1-q_2)\Bigl( \bigl\vert c_m(n_{t-1}^m)-u_{t-1}^m\bigl|\, +\, q_2 \bigl|c_m(n_{t-2}^m)-u_{t-2}^m\bigl|\, +\, \dotsb \\
				& \mbox{\hspace{10em}} + q_2^{t-1} \bigl|c_m(n_{1}^m)-u_{1}^m\bigl| \Bigr),
	\end{split}
	\end{equation}
	which can be implemented recursively as
	\begin{equation}\label{eq:uvrec}
	\begin{split}
		u_t^m & = q_1 u_{t-1}^m + (1-q_1)c_m(n_{t-1}^m)\\
		v_t^m & = q_2 v_{t-1}^m + (1-q_2)\,\bigl|c_m(n_{t-1}^m)-u_{t-1}^m\bigl|,
	\end{split}
	\end{equation}
	thus requiring minimal storage at the central authority.

	\subsection{Agents and their Policies}
	Agents base their decisions at time $t$ on the information provided by the central authority at time $t$ by applying their policies. The heterogeneity of the agent population is taken into account by letting each agent have its own policy. Of course, a lot of policies are imaginable; we consider here a set of policies that is parametrized by $\omega\in [0,1]$, a weight in each agent's cost function: An agent using policy $\omega$ will choose the resource $m^\star$ that minimizes $\omega u_t^m + (1-\omega) v_t^m$. In other words, at time $t$ agents with policy $\omega$ select resource 
	\begin{align}
		\pi^\omega(s_t) \coloneq \arg \min_{m=1,\dotsc,M} \omega u_t^m + (1-\omega) v_t^m.
\label{piomega}
	\end{align}
	If for instance, $u_t^m$ (resp.\ $v_t^m$) is the minimal (resp.\ maximal) cost of resource $m$ observed in the last $r$ time steps, then agents with $\omega=1$ entirely disregard the worst case and could be said to be risk-seeking; similarly, $\omega=0$ would be risk-averse and $\omega=0.5$  risk-neutral. If instead a mean and toll are broadcast, $\omega$ parametrizes the trade-off between each agent's travel time and their wallet.
		
	Of course, the population of agents in a road network changes over time. 
 Mare\v{c}ek et al. \cite{marevcek2016signaling} modelled this by assuming that 1) the number $N$ of agents is fixed and 2) the distribution of policies $\omega$ among the agents is i.i.d., i.e.\ at each time step, the distribution is selected from a finite set of distributions $\{\eta_1,\dotsc,\eta_K\}$ according to a distribution $d=(d_1,\dotsc,d_K)$. A little more formally: at each $t$, choose $\mu_t$ from $\{\eta_1,\dotsc,\eta_K\}$ according to $\prob(\mu_t=\eta_k)=d_k$; the population of agents at time $t$ is then $\mu_t$.  

\subsubsection{Population Dynamics}	\label{ssec:popdyn}
Both assumptions -- constant number of agents and i.i.d.\ replacement of the population at every $t$ -- are not very realistic. Here, we replace the assumption of i.i.d.\ population renewal (according to the fixed distribution $d$) with a Markov chain that makes $d$ dynamic; in this way the new population selected at each time step can depend on the previous population. For instance, it seems unlikely that a population of mainly risk-averse drivers at one time step will give rise to a population of mainly risk-seeking ones in the next instance, reflecting the fact that not all users leave the road network at once, or that there might be some correlation between, say, the time of day and what kind of drivers are on the road. 

More precisely: As before, we assume that there is a family of possible distributions of the levels of risk aversion in the population, with a finite index set $\J=\{1,\dotsc,K\}$. Additionally, there is a Markov chain with $K$ states and transition probability matrix $P\in [0,1]^{K\times K}$. The probability of appearance of population $\eta_j$, $j \in \J$,  at iteration $t+1$ is now given by  $\prob(\mu_{t+1}=\eta_j\mid \mu_t = i) = p_{i j}$, i.e.\ the probability of a specific $\eta_j$ depends on what the last observed population was.	This setting is clearly more general: with $p_{ij}=1/K$ $\forall i,j\in\J$, we recover the original i.i.d.\ assumption. 
 
The transition probabilities $p_{ij}$ can be used to model different ideas about the mechanism behind the change in populations. Two straightforward ones are reviewed briefly here, more details are given in the appendix. 

The Markov chain could be used to encode dependence of the population on the time of day. As a simple example, say that there is a morning, noon, evening and night population of drivers, perhaps many risk-seeking drivers at night, mostly risk-neutral commuters in the morning, etc. Say those populations for morning, noon, evening, and night are numbered as 1,2,3, and 4. The corresponding transition matrix and adjacency graph are:\\
\parbox{.4\textwidth}{\centering
\[
			P = \bmat{0&1&0&0\\0&0&1&0\\0&0&0&1\\1&0&0&0}
\]
}
\parbox{0.6\textwidth}{\centering
\includegraphics[clip=true,width=0.6\textwidth]{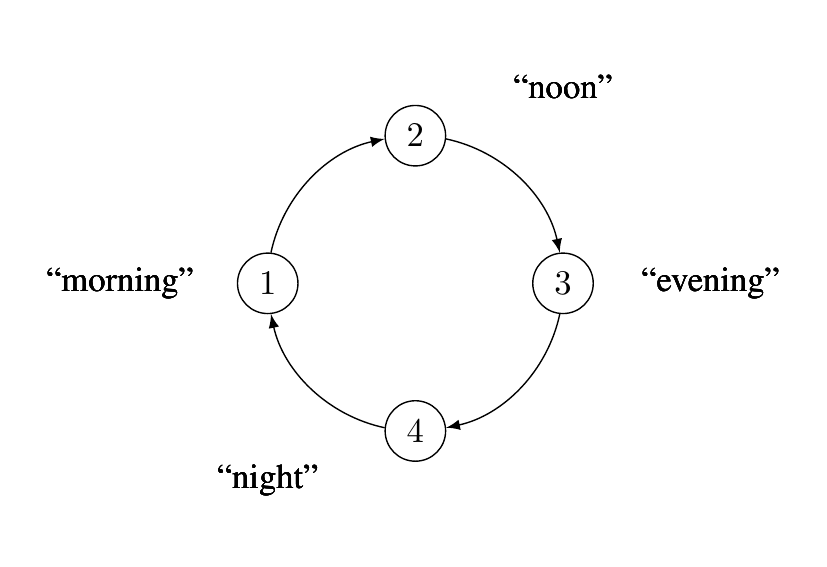} 
}
A simple generalization allowing for more than just one population per time of day can be achieved by replacing the ``1''s by appropriate blocks, see Section~\ref{ssec:app:timedep} for details.

If instead we are interested in a gradual change in population, e.g.\ due to individual drivers changing their behaviour or leaving and entering the road network, then the probability of the next population being $\eta_j$ given that the current one is $\eta_i$ should depend on ``how different'' they are from each other. In more mathematical terms, let $\Delta(\cdot,\cdot)$ be a metric on the space of populations so that $\Delta(\eta_i,\eta_j)$ provides a notion of distance between the populations $\eta_i$ and $\eta_j$. Then the probability $p_{ij}$ should be a decreasing function of $\Delta(\eta_i,\eta_j)$: the farther apart $\eta_i$ and $\eta_j$ are, the less likely it is that one gives rise to the other in the next time step. To this end, we choose a parameter $\psi\in(0,1)$ that reflects the probability of an agent changing its policy, and let $p_{ij}$ proportional to $\psi^{\Delta(\eta_i,\eta_j)}$; see Section~\ref{ssec:app:EMD}.

In the case of drivers changing their policies, more drastic changes (e.g.\ going from risk-averse to risk-neutral) should contribute more to the distance than slight changes; this can be described by a discrete version of the \textit{Wasserstein metric} (see Section~\ref{ssec:app:EMD}). If, however, we want to model drivers leaving the road network and being replaced, then the distance $\Delta(\eta_i,\eta_j)$ should depend only on how many drivers have to change policy to transform $\eta_i$ into $\eta_j$ (and not on the size of those policy changes); let us call this metric the \textit{substitution metric}, since that is what it describes. Both of those notions are captured as special cases of the so-called Earth Mover's Distance (EMD, c.f. \cite{RubnerTomasiGuibas2000}). 
 For details such as definitions and computation, see Section~\ref{ssec:app:EMD}. 



 \skipthis{$P$ can be reflecting different populations at different times of day, in which case shift matrix (possibly replace 1s by blocks), or gradual changes in population due to drivers changing their minds or leaving and entering road network.\\ 		
 Latter case  requires notion of distance between populations: The ``closer'' $\eta$ is from $\gamma$, the more likely it is to follow $\gamma$. Here introduce only generic distance $\Delta(\eta,\gamma)$, defer specifics (Wasserstein, $\psi^{\Delta(\eta,\gamma)}$, etc.) to an appendix?}

\section{Iterated Function Systems (IFS)}\label{sec:ifs}
An \textit{iterated function system} is a generalization of a Markov chain. It consists of a state space $X$ with its metric $d$, a family $W$ of Lipschitz functions\footnote{%
	A function $f$ on the metric space $(X,d)$ is Lipschitz with constant $s$, or ``$s$-Lipschitz,'' if for all $x,y\in X$, we have $d\bigl(f(x),f(y)\bigr) \leq s d(x,y)$.
} $W=\{w_j:X\rightarrow X \, | \, j\in \J\}$, where $\J$ is some index set (finite, countably infinite or worse, but let us assume here that it is countable) and a measure $\nu$ that makes $(\J,\cdot,\nu)$ a probability space. 

At each iteration $t$ of the IFS, $j$ is selected from $\J$ according to $\nu$ and $w_j$ is applied to the current state $x_t$ to obtain $x_{t+1}$. Formally:
\[
\prob(X_{t+1}\in A | X_t=x_t) \coloneq \sum_{\J} \ind_{\{i | w_i(x_t)\in A  \}}(j) \nu(j),
\]
i.e.\ the probability of $x_{t+1}$ ending up in a set $A$ is the probability of selecting an index $j$ such that $w_j(x_t)$ is in $A$ (the measure of the set of indices $j$ for which $w_j(x_t)$ is in $A$). Here the Markov property is clear: the distribution of the next state $X_{t+1}$ depends only on the current state $x_t$ and not any ``older'' states $x_{t-1}$ etc.

In this way, the IFS ``jumps'' around $X$. Unless we have a degenerate case such as all $w_j$ having the same fixed point, we can not expect the sequence $\{x_t\}_{t=0}^{\infty}$ to converge in a classical sense; instead, we can have a weaker form, \textit{convergence in distribution}: that there is a distribution $\Pi$ on $X$ such that as $n\rightarrow\infty$, the set $\{x_0,x_1,\dotsc,x_n\}$ will be distributed according to $\Pi$.

\begin{theorem}[E.g. \protect{\cite[Thm. 1.1]{DiaconisFreedman1999}}] \label{thm:ifs}
	Let $L_j$ denote the Lipschitz constant of $w_j$ and assume that the IFS is \emph{contractive on average}, i.e.
	\begin{equation}
	\sum_{\J} \nu(j) \log(L_j)<0. \label{eq:contronavg}
	\end{equation}
	Then, there is a distribution $\Pi$ on $X$ such that $\{x_0,x_1,\dotsc,x_n\}$ is distributed according to $\Pi$ as $n\rightarrow\infty$.
\end{theorem}
If $W$ is a family of contractions, i.e.\ if $L_j<1$ for all $j$, then \eqref{eq:contronavg} is trivially satisfied.

\paragraph{IFS theory and congestion control} The traffic model described above can be recast as an IFS, where the index set is $\J = \{1,\dotsc,K\}$, and the randomly chosen function applied to the congestion profile $n_t$ at each time step is parametrized by the population distribution $\eta_j$. Theorem~\ref{thm:ifs} can then be applied to make the statement that there is a distribution $\Pi$ such that, as $t\rightarrow\infty$, $n_t$ follows this distribution. In other words, $n_t$ will not be totally erratic, but it will be predictable in the sense that it will eventually behave like samples from a random variable with a fixed distribution. 

\subsection{Recurrent IFSs}

The assumption in e.g.~\cite{marevcek2016signaling} of i.i.d.\ population renewal at each time step was necessary in order to apply  Theorem~\ref{thm:ifs}, as the probability measure $\nu$ is not allowed to change over time or depend on the state $x$. However, extensions that relax this do exist, e.g.~ \cite{MauldinSzarekUrbanski2009,barnsley1988invariant}. Here, we consider the extension in \cite{BarnsleyEltonHardin1989}, where a \textit{recurrent iterated function system} (RIFS) is introduced as an IFS with an underlying Markov chain that modifies $\nu$ at each time step. More precisely:

We have an IFS as described in the last section with a finite index set $\J$, say $\J=\{1,\dotsc,K\}$. Additionally, there is a Markov Chain with $K$ states and transition probability matrix $P\in [0,1]^{K\times K}$. The probability of applying $w_j$ at iteration $t+1$ is now given by  $\prob(i_{t+1}=j|i_t) = p_{i_t j}$, i.e.\ the probability of applying a specific $w_j$ depends on what the last applied function $w_{i_t}$ was! This is in contrast to the case of Section~\ref{sec:ifs}, where the probability to select a specific $w_j$ was always the same and given by $\nu(j)$.

This setting is clearly more general: for a classical IFS, we can simply choose $p_{ij}=1/K$ $\forall i,j\in\J$. On the other hand, the way that $X_t$ jumps around in $X$ now is \emph{not} a Markov process anymore --- the distribution of $X_t$ not only depends on $X_{t-1}$, but also on $i_{t-1}$ --- but the joint process of $(X_t,i_t)$ jumping around in $X\times \J$ is.

Results analogous to Theorem~\ref{thm:ifs} can be stated for this case, see e.g.\ \cite{barnsley1988invariant,BarnsleyEltonHardin1989}. We state
\begin{theorem}[\cite{BarnsleyEltonHardin1989}]
	\label{thm:markch}
	Assume we have an RIFS as described above, and let $m:\{1,\dots,K\}\rightarrow[0,1]$ denote the stationary distribution of the underlying Markov chain (i.e.\ $m$ corresponds to the normalized Perron eigenvector of $P^T$). Then, if
	\begin{equation}\label{eq:avgcontractive}
		\sum_{i=1}^K m(i) \log L_i = E_m\{ \log L_{i}  \} <0,
	\end{equation}
	there is a unique stationary distribution $\tilde{\nu}$ of the Markov process $(X_t,i_t)$ and $X_t$ converges in distribution to $\nu$ with $\nu(B) = \tilde{\nu}(B\times\J)$. Here, $L_i$ again denotes the Lipschitz constant of $w_i$, and $E_m$ denotes expected value with respect to $m$. 
\end{theorem}
\begin{proof}
	This is just a corollary (much weaker, but sufficient for our purposes) to \protect{\cite[Thm.\ 2.1 (ii)]{BarnsleyEltonHardin1989}}, which follows by taking $n=1$ and removing the specifics of the stationary distributions.\qed
\end{proof}

If~\eqref{eq:avgcontractive} holds, we again say that the RIFS is \textit{average contractive} or \textit{contractive on average}.

That means that the model in Section~\ref{sec:mod} can be made more general while staying amenable to a similar stability analysis. 

\section{The Main Result}
\label{sec:main}

The traffic model described above can be recast as iterated applications 
of functions to the broadcast signals $s_t$, where the function to apply is
chosen from a family whose index set is $\J = \{1,\dotsc,K\}$, with the next population of agents depending on the current population.
At each iteration $t$, $i_{t+1}$ is selected from $\J$ according to the $i_t$-th row of $P$, and the computation of the signals and the response of the population is captured by $w_{i_{t+1}}$ being applied to the current signal $s_t$ to obtain $s_{t+1}$. That, and under which conditions, we can hope for convergence in distribution of the signalling process and the congestion profile is stated in the following theorem:
\begin{theorem}\label{thm:main}
There exists a set of constants $\kappa_m'$, $m=1,\dotsc,M$ such that if the costs $c_m(\cdot)$ are $1/\kappa_m'$-Lipschitz (with respect to the 1-norm), the signal $s_t$ and the congestion profile $n_t$ converge in distribution as $t\rightarrow\infty$.
\end{theorem}
\begin{proof}\ \\
\underline{Step 1:}
We begin by writing the signalling process $\{s_{t}\}$ generated by the model described in Section~\ref{sec:mod} and Figure~\ref{fig:modflow} as an IFS on the state-space $X=\reals_+^{2M}$ along similar lines as in~\cite{marevcek2016signaling}. Recall that
\[
	s_t = \bmat{u_t^1 & v_t^1 & u_t^2 & \dotsm &u_t^M & v_t^M}
\]
and assume the signalling scheme described in~\eqref{eq:uv} and~\eqref{eq:uvrec}; we repeat for convenience:
\begin{equation}\label{eq:uvrec2}
	\begin{split}
	u_t^m & = q_1 u_{t-1}^m + (1-q_1)c_m(n_{t-1}^m)\\
	v_t^m & = q_2 v_{t-1}^m + (1-q_2)\,\bigl|c_m(n_{t-1}^m)-u_{t-1}^m\bigl|.
	\end{split}
\end{equation}

Now observe that
\begin{align}
	u_{t+1}^m &= q_1 u_{t}^m + (1-q_1)c_m(n_{t}^m), \label{eq:ut1} \\
	\intertext{but }
	n_{t}^m &= \card\left(\left\{ i
								 \mid a_t^i=m \right\}\right) 
		= 	\sum_{i=1}^N \delta_{a_t^i,m}  
		=	\sum_{i=1}^N \sum_{\omega\in\Omega}  
				\delta_{\omega^i,\omega} \delta_{\pi^\omega(s_t),m} \notag\\
		& = \sum_{\omega\in\Omega} \underbrace{\left(\sum_{i=1}^N \delta_{\omega^i,\omega} \right)}
				\ind_{\{ \sigma \mid \sigma u_t^m + (1-\sigma)v_t^m < \sigma u_t^s + (1-\sigma)v_t^s \;\forall s\neq m \} }(\omega) \notag \\
		& = \sum_{\omega\in\Omega} \makebox[5.5em][c]{$\mu_t(\omega)$}  \ind_{\{ \sigma \mid \sigma u_t^m + (1-\sigma)v_t^m < \sigma u_t^s + (1-\sigma)v_t^s \;\forall s\neq m \} }(\omega),		\label{eq:fcnmtnt}
\end{align}
where
\[
	 \delta_{a_t^i,m} = \ind_{\{j \mid a_t^j=m\}}(i) = 
			 \begin{cases} 1 & \text{ if } a_t^i=m \\ 0 & \text{ else.}\end{cases}
\]
is the Kronecker symbol, so $\delta_{\omega^i,\omega}$ is 1 if agent $i$ has policy $\omega$, and $\delta_{\pi^\omega(s_t),m}$ is 1 if policy $\omega$ selects resource $m$.  Now, \eqref{eq:fcnmtnt} is unwieldy, but close inspection reveals that it shows that for fixed $\mu_t$, $n_t^m$ is just a function of $s_t$ (through $u_t^s$ and $v_t^s$, $s=1,\dotsc,N$), and plugging this into~\eqref{eq:uvrec2}, we see that $s_{t+1}$ can be written as a function of $s_t$, parametrized by the random variable $\mu_t$:
\[
	s_{t+1} = w_{j}(s_t) \text{ if } \mu_t=\eta_j.
\]
The size of the index set $\J$, i.e.\ the number $K$ of possible distributions, is likely gigantic, namely
	 \[
		 K = \binom{\card(\Omega)+N-1}{N},
	 \]
but nevertheless finite.

\underline{Step 2:} Next, it will be established that if the $c_m$ are Lipschitz, then so are the $w_{j}$ of Step 1. For this, let $x,y\in\reals_+^{2M}$, $x\neq y$ (they play the role of two possible $s_{t-1}$) and fix the population $\mu_t$ to some $\eta_j$. From~\eqref{eq:fcnmtnt} it then follows that the congestion profile at time $t$ is just a function of the signal broadcast at $s_{t-1}$, i.e.\ $x$ or $y$. Let us write $n_t^m(x)$ to reflect this. 
To investigate the Lipschitz constant of $w_j$, we write
\begin{equation*}
	\norm{w_j(x) - w_j(y)}_1 = \sum_{m=1}^{M} R_1^m + R_2^m,
\end{equation*}
where
\begin{multline*}
	R_1^m  \coloneq \Bigl\lvert q_1 x_{2m-1} + (1-q_1) c_m \bigl(n_t^m(x)\bigr) -  \left[q_1 y_{2m-1} + (1-q_1) c_m \bigl(n_t^m(y)\bigr)\right]\Bigr\rvert
\end{multline*}
and
\begin{multline*}
	 R_2^m \coloneq  \Biggl\lvert  q_2 x_{{2m}} + (1-q_2)\abs{c_m(n_{t-1}^m(x))-x_{2m-1}}
	 \\ -  \Bigl[
	q_2 y_{{2m}} + (1-q_2)\abs{c_m(n_{t}^m(y))-y_{2m-1}} 
	\Bigr]		\Biggr\rvert	
\end{multline*}
where the first term in the sum corresponds to $u_{t+1}^m$ and the second to $v_{t+1}^m$. We shall investigate those two terms separately, but first we repeat an important result from the proof of Theorem 1, specifically equation (15), 
 in~\cite{marevcek2016signaling}, namely that $n_t^m(\cdot)$ is Lipschitz in the 1-norm, so there is a constant $\kappa_m$ such that 
\begin{equation}\label{eq:ntm-lip}
	\abs{n_t^m(x) - n_t^m(y)} \leq \kappa_m \max_\omega \mu(\omega) \norm{x-y}_1 \leq \kappa_m N \norm{x-y}_1.
\end{equation}
With~\eqref{eq:ntm-lip}, we then have 
\[
	\abs{c_m\bigl(n_t^m(x)\bigr)-c_m\bigl(n_t^m(y)\bigr)}\leq L_m 	\abs{n_t^m(x) - n_t^m(y)} \leq L_m \kappa_m N \norm{x-y}_1,
\]
where $L_m$ denotes the Lipschitz constant of $c_m$.

Now we can bound
\begin{multline*}
	R_1^m = 
	\Bigl\lvert q_1 (x_{2m-1}- y_{2m-1}) + (1 - q_1) c_m \bigl(n_t^m(x)\bigr) -  c_m \bigl(n_t^m(y)\bigr)\Bigr\rvert \leq \\
	  q_1 \bigl\lvert x_{2m-1}- y_{2m-1}\bigr\rvert + (1 - q_1) L_m \kappa_m N \norm{x-y}_1
\end{multline*}
and
\begin{multline*}
	R_2^m = \\ \Bigl\lvert  q_2 (x_{2m}- y_{2m}) + (1-q_2)
		\bigl(\abs{c_m(n_{t-1}^m(x))-x_{2m-1}}-\abs{c_m(n_{t-1}^m(y))-y_{2m-1}}
	  \bigr)\Bigr\rvert \leq \\
	  q_2 \abs{x_{2m}- y_{2m}} + (1-q_2) \bigl(	\abs{x_{2m-1}-y_{2m-1}}+
			  \abs{c_m(n_{t-1}^m(x)) - c_m(n_{t-1}^m(y))}	  \bigr)\leq \\
	  q_2 \abs{x_{2m}- y_{2m}} + (1-q_2) \abs{x_{2m-1}-y_{2m-1}} +
		   (1-q_2)L_m \kappa_m N  \norm{x-y}_1,
\end{multline*}
where we used that $\abs{a-b}-\abs{a'-b'}\leq\abs{a-a'}+\abs{b-b'}$, which can be derived using the reverse triangle inequality.\footnote{$
	|a-b|-|a'-b'|\leq\Bigl\lvert|a-b|-|a'-b'|\Bigr\rvert\leq|a-b-(a'-b')|=|a-a'+b'-b|\leq|a-a'|+|b-b'|
	$}
Reassembling and slight reordering of terms yields
\begin{multline}
	\norm{w_j(x) - w_j(y)}_1 = \sum_{m=1}^{M} R_1^m + R_2^m \leq \\
	\sum_{m=1}^{M} (1 + q_1 - q_2)\abs{x_{2m-1}- y_{2m-1}} + q_2 \abs{x_{2m}- y_{2m}} + 
			(2-q_1-q_2) L_m\kappa_m N \norm{x-y}_1 \leq \\
	\underbrace{ \Biggl( \max\{q_2, 1+q_1-q_2\} + (2-q_1-q_2)N \sum_{m=1}^{M} L_m\kappa_m
	\Biggr)    }_{\eqcolon \bar{L}_\mu} \norm{x-y}_1
\end{multline}
and so we have found an expression for (an upper bound of) the Lipschitz constant $\bar{L}_j$ of $w_j$.

\underline{Step 3:} Now we are in a position to formulate sufficient conditions on $L_m$ to allow application of Theorem~\ref{thm:markch}. The condition 
\[
	\sum_{j\in\J} m(j) \log \bar{L}_j <0
\]
is certainly satisfied independently of $m$ (recall that $m$ is the stationary distribution of the underlying Markov chain) if $\bar{L}_j<1$ for all $j$. This is the case if 
\begin{equation}\label{eq:q2}
	 \sum_{m=1}^{M} L_m\kappa_m < \frac{1- \max\{q_2, 1+q_1-q_2\}}{N\,(2-q_1-q_2)}
\end{equation}
Hence the $\kappa_m'$ postulated in the theorem statement need to satisfy
\[
	\sum_{m=1}^{M}\frac{\kappa_m}{\kappa_m'} < \frac{1- \max\{q_2, 1+q_1-q_2\}}{N\,(2-q_1-q_2)}.
\]
One choice is 
\begin{equation}\label{eq:kappam}
	\kappa_m' = \kappa_m M \frac{N\,(2-q_1-q_2)}{1- \max\{q_2, 1+q_1-q_2\}},
\end{equation}
but there are obviously many more.

Hence, step 1 shows that the signalling process $\{s_t\}$ can be cast as an RIFS on the normed space $(\reals^{2M},\norm{\cdot}_1)$ with a family of functions $\bigl\{w_j \mid j\in\J
\bigr\}$, that by step 2 are Lipschitz. By step 3, under the right conditions, this RIFS is contractive on average and by Theorem~\ref{thm:markch}, we then conclude that $s_t$ converges in distribution. Since, by~\eqref{eq:fcnmtnt}, $n_t$ is a function of the two random variables $s_t$ and $\mu_t$, the congestion profile also converges in distribution. \qed
\end{proof}

\begin{remark}
	Of course, Lipschitz constants have to be nonnegative, and thus~\eqref{eq:q2} implies that in order for the proof to work, $q_2>q_1$ has to hold. 
\end{remark}
\begin{remark}
	While it is necessary for $c_m(\cdot)$ to be Lipschitz in order for the theory of RIFSs to apply, the bounds in~\eqref{eq:kappam} are far from tight, due to conservative approximations made in the proof. For instance in~\eqref{eq:ntm-lip}, $\max_\omega \mu(\omega)$ is bounded by $N$. This bound is obviously only tight for populations where all agents have the same policy; for most populations, it is a gross overestimation. Additionally, numerical computations of $m$, the stationary distribution of the Markov chain, indicate that the stationary probability $m(i)$ of such populations, and hence their weight $m(i)$ in~\eqref{eq:avgcontractive}, is typically very small, and this bound could most likely be improved upon.
\end{remark}

	

\section{A Computational Illustration}
\label{sec:computational}

\begin{figure}[t!]
  \includegraphics[clip=true,width=0.49\textwidth]{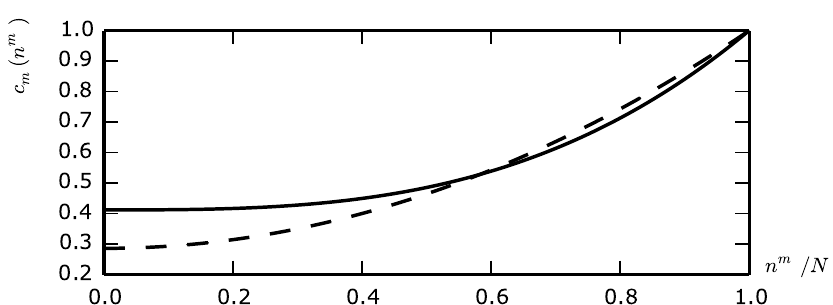}
  \includegraphics[clip=true,width=0.49\textwidth]{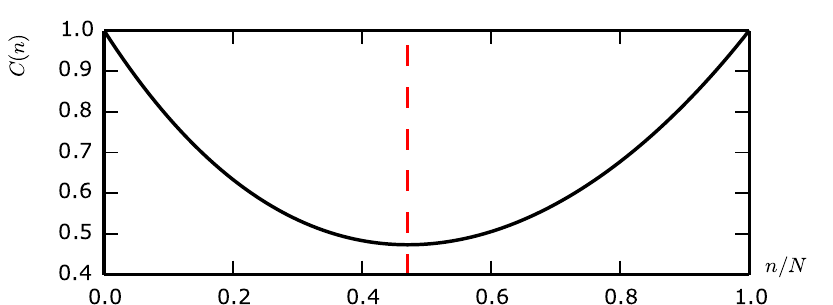}              \\
  \includegraphics[clip=true,width=0.49\textwidth]{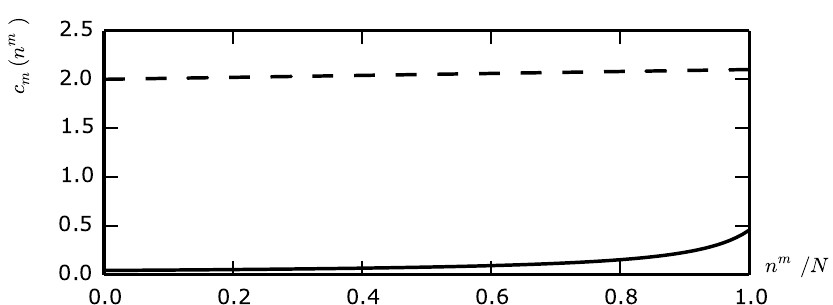}
  \includegraphics[clip=true,width=0.49\textwidth]{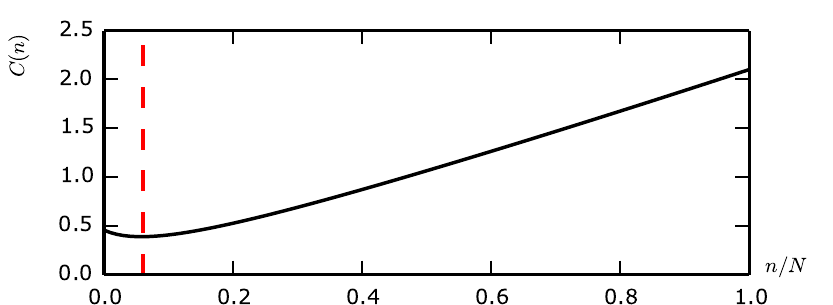}              \\
  \includegraphics[clip=true,width=0.49\textwidth]{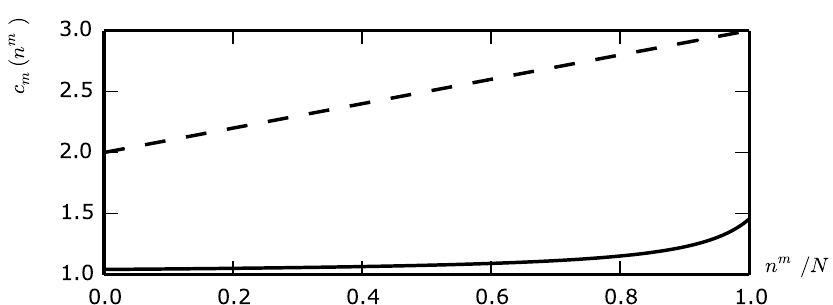}
  \includegraphics[clip=true,width=0.49\textwidth]{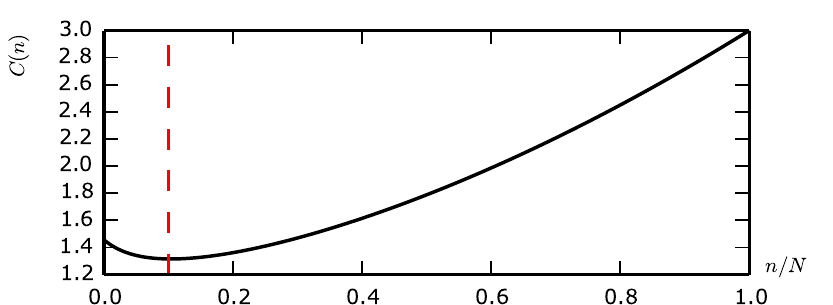}              
  \caption{Three examples. 
    Top left: Cost functions $x^2 + 0.4$ (dashed) and $(x^3 + 0.7)/1.7$ (solid line).
    Top right: The corresponding social cost with optimum of approximately 0.473 at 0.47 (dashed vertical line).
    Middle left: Cost functions $x/10 + 2$ (dashed) and $1 + 1/(1.1-x)/22$ (solid line).
    Middle right: The corresponding social cost with optimum of approximately 0.387 at 0.06 (dashed vertical line).
    Bottom left: Cost functions $x + 2$ (dashed) and $1 + 1/(1.1-x)/22$ (solid line).
    Bottom right: The corresponding social cost with optimum of approximately 1.315 at 0.1 (dashed vertical line).
}
  \label{fig:ushaped4settings}
\end{figure}

Let us now illustrate the main result on a number of examples. Throughout, we use $M=2$,
three variants of $c_1(x)$ and $c_2(x)$ as suggested in Figure~\ref{fig:ushaped4settings},
and $q_1 = 0.45$, $q_2 = 0.5$.
First, consider an example with $N=2$ agents and $\Omega=\{0,0.5,1\}$, i.e., $\card(\Omega)=3$ policies and $K=\binom{\card(\Omega)+N-1}{N}=6$ possible distributions, 
assuming that every distribution possible in theory is also sensible in practice.
One possible transition matrix for the case of the Wasserstein metric is, up to rounding, 
\[
	P_W = \bmat{
		0.35 & 0.18 & 0.12 & 0.12 & 0.07 & 0.06 \\ 
		0.22 & 0.28 & 0.18 & 0.18 & 0.11 & 0.09 \\ 
		0.14 & 0.18 & 0.29 & 0.12 & 0.18 & 0.14 \\ 
		0.14 & 0.18 & 0.12 & 0.29 & 0.18 & 0.14 \\ 
		0.09 & 0.11 & 0.18 & 0.18 & 0.28 & 0.22 \\ 
		0.06 & 0.07 & 0.12 & 0.12 & 0.18 & 0.35},
\]
whereas the corresponding transition matrix for the substitution metric is
\[
	P_S = \bmat{
		0.44 & 0.14 & 0.07 & 0.14 & 0.06 & 0.07 \\ 
		0.18 & 0.36 & 0.18 & 0.14 & 0.14 & 0.07 \\ 
		0.07 & 0.14 & 0.44 & 0.06 & 0.14 & 0.07 \\ 
		0.18 & 0.14 & 0.07 & 0.36 & 0.14 & 0.18 \\ 
		0.07 & 0.14 & 0.18 & 0.14 & 0.36 & 0.18 \\ 
		0.07 & 0.06 & 0.07 & 0.14 & 0.14 & 0.44}.
\]

For instance, this reflects that in the former case there is a 28-35\% chance of the population remaining the same, whereas in the latter, this chance is 36-44\%. For details on how these matrices are obtained, we again refer to Appendix~\ref{sec:app:P}. 
The corresponding behaviour is captured in Figure~\ref{fig:ushaped4withN2},
where each time-series has been obtained as the mean over 10,000 sample paths\footnote{
Notice that although the horizon may seem short and the number of sample paths may seem low, this is justified by our main result, i.e., the convergence in distribution.
Indeed, one may argue that convergence in distribution allows for a principled use of simulations.
}, with error bars at one standard deviation.
Notice that in the case of $M = 2$, the social cost $\sum_{m = 1}^{M} (n_t^m / N) c_m(n_t^m)$ is determined by $n_t^1$ alone, justifying the notation $C(n_t^1)$ .  
Throughout, as proven, one observes the convergence in distribution.

\begin{figure}[t!]
  \includegraphics[clip=true,width=0.49\textwidth]{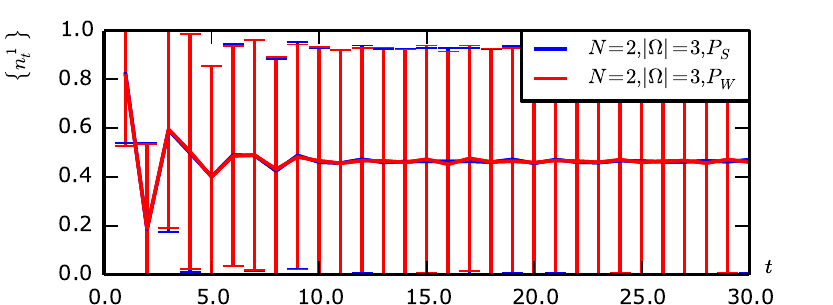}
  \includegraphics[clip=true,width=0.49\textwidth]{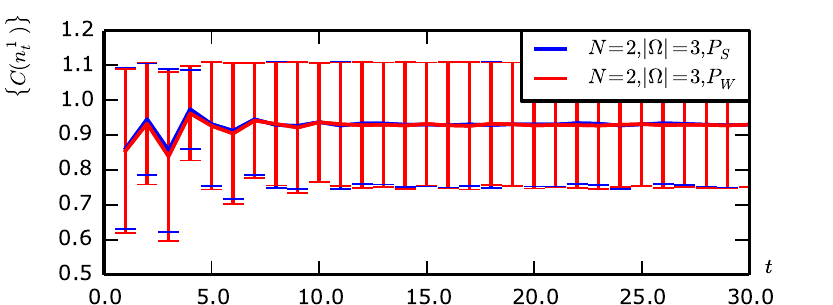} \\
  \includegraphics[clip=true,width=0.49\textwidth]{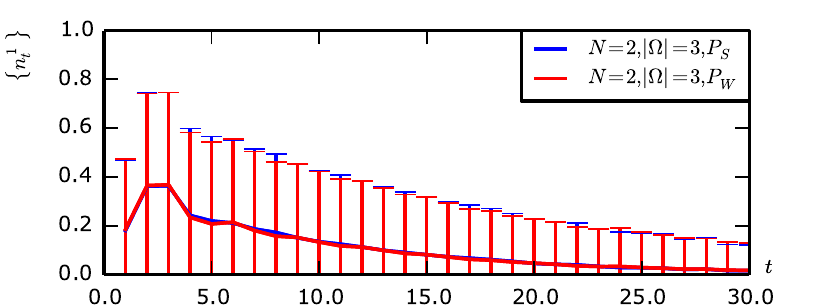}
  \includegraphics[clip=true,width=0.49\textwidth]{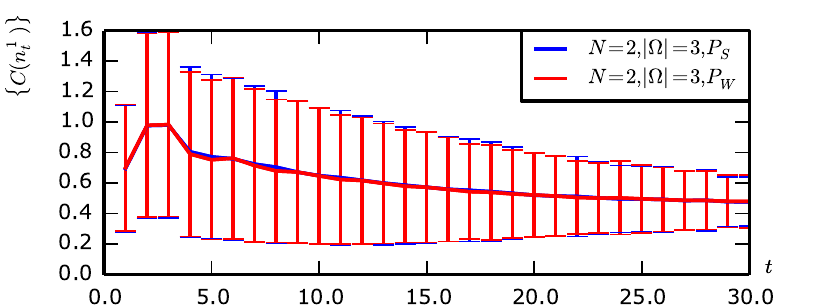} \\
  \includegraphics[clip=true,width=0.49\textwidth]{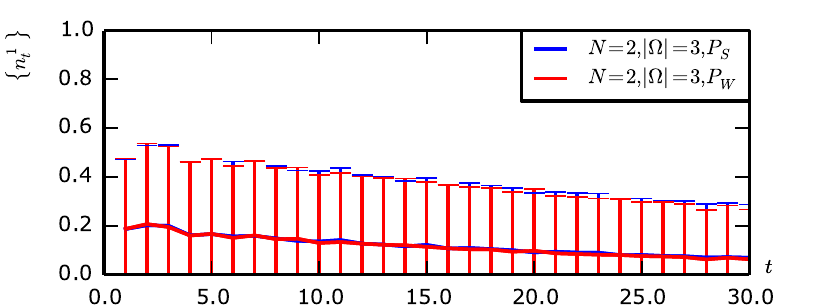}
  \includegraphics[clip=true,width=0.49\textwidth]{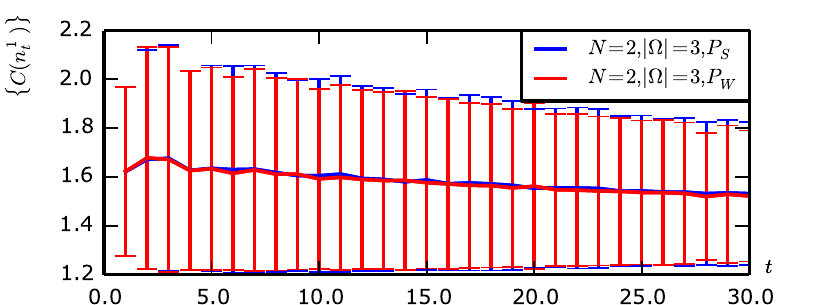}              
  \caption{The behaviour of $N = 2$ agents on the three examples of Figure~\ref{fig:ushaped4settings}, considering both the Wasserstein metric (in red) and the discrete metric (in blue).
    Top left: Evolution of state over time for $x^2 + 0.4$ and $(x^3 + 0.7)/1.7$.
    Top right: The corresponding evolution of the social cost.
    Middle left: Evolution of state over time for $x/10 + 2$ and $1 + 1/(1.1-x)/22$.
    Middle right: The corresponding evolution of the social cost.
    Bottom left: Evolution of state over time for $x + 2$ and $1 + 1/(1.1-x)/22$.
    Bottom right: The corresponding evolution of the social cost. 
}
  \label{fig:ushaped4withN2}
\end{figure}

\begin{figure}[t!]
	\centering
	\parbox[t]{.45\textwidth}{\centering \includegraphics[width=0.44\textwidth]{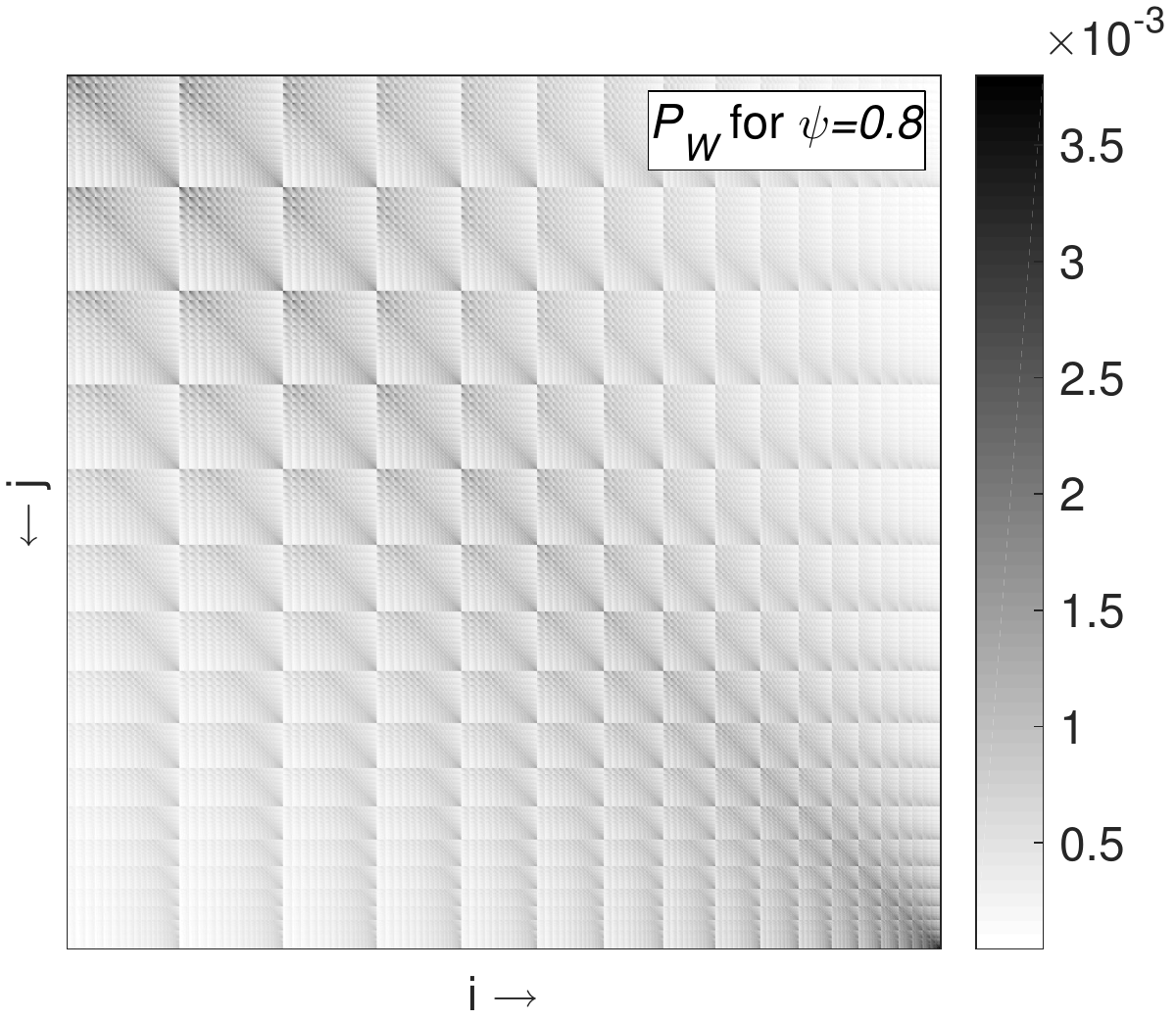}}
	\hfill
	\parbox[t]{.45\textwidth}{\centering \includegraphics[width=0.44\textwidth]{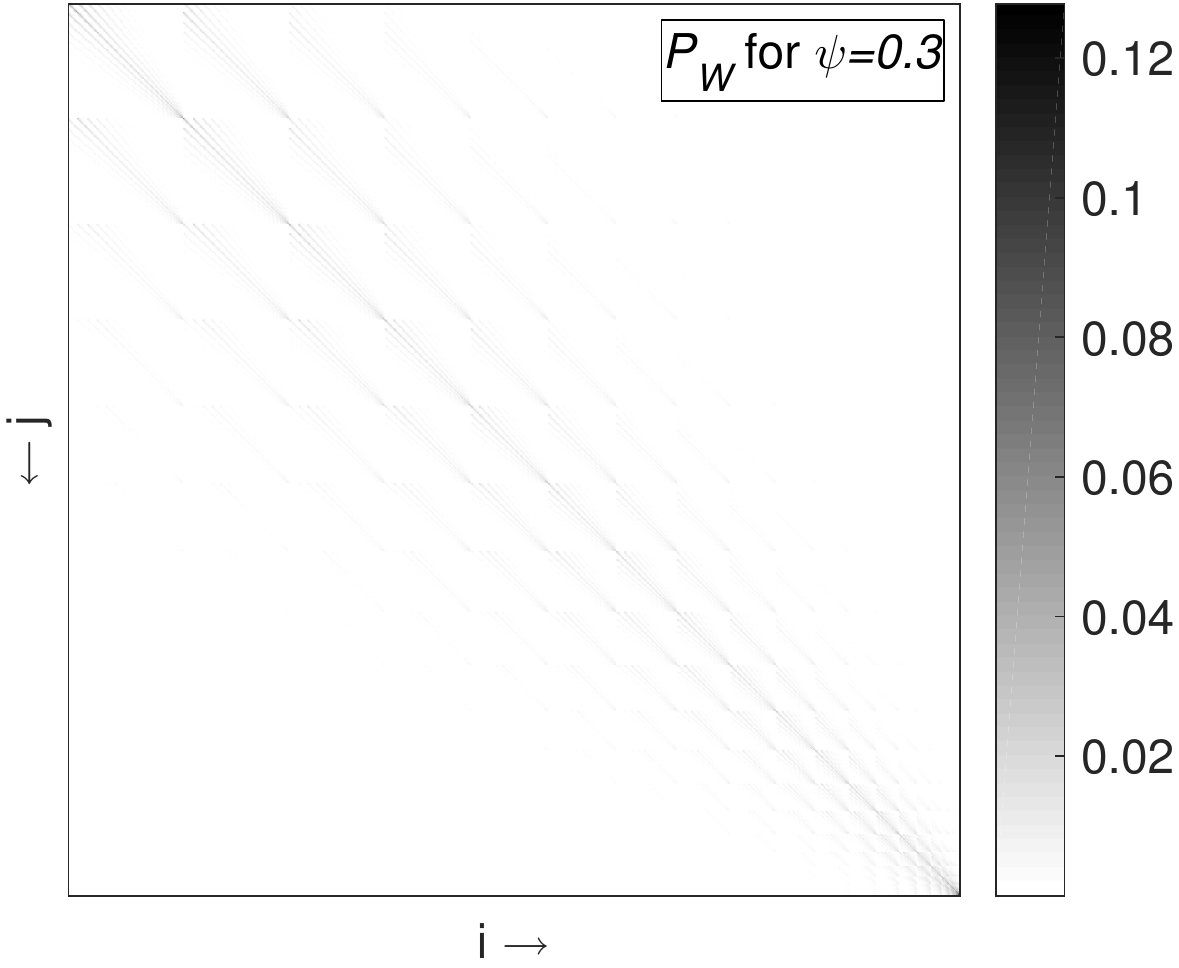}}
	\caption{Transition matrices according to the Wasserstein metric with $N=20$ agents and $\Omega=\{0,1/3,2/3,1\}$ for two different values of $\psi$. The apparent fractal structure is an artefact of the order chosen for the populations $\eta_i$ (Technically, this is referring to an order relation on the set $\{0,\dotsc,N\}^{\card(\Omega)}$, in which the possible populations make up the subset of tuples summing to $N$. The lexicographic order is chosen for the figures, however the effect of a different choice is merely cosmetic.) }
	\label{fig:Pwass}
	\parbox[t]{.45\textwidth}{\centering \includegraphics[width=0.44\textwidth]{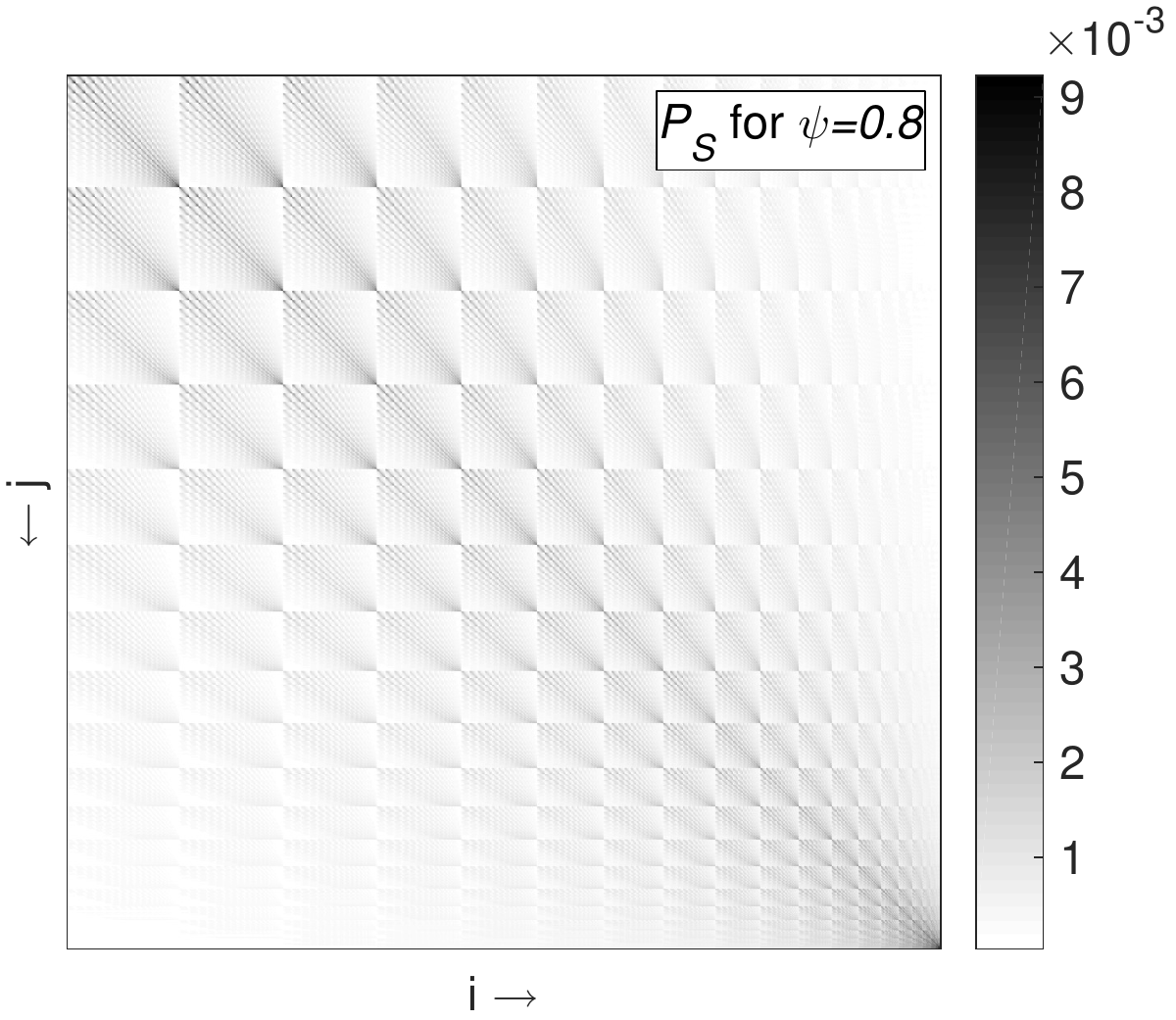}}
	\hfill
	\parbox[t]{.45\textwidth}{\centering \includegraphics[width=0.44\textwidth]{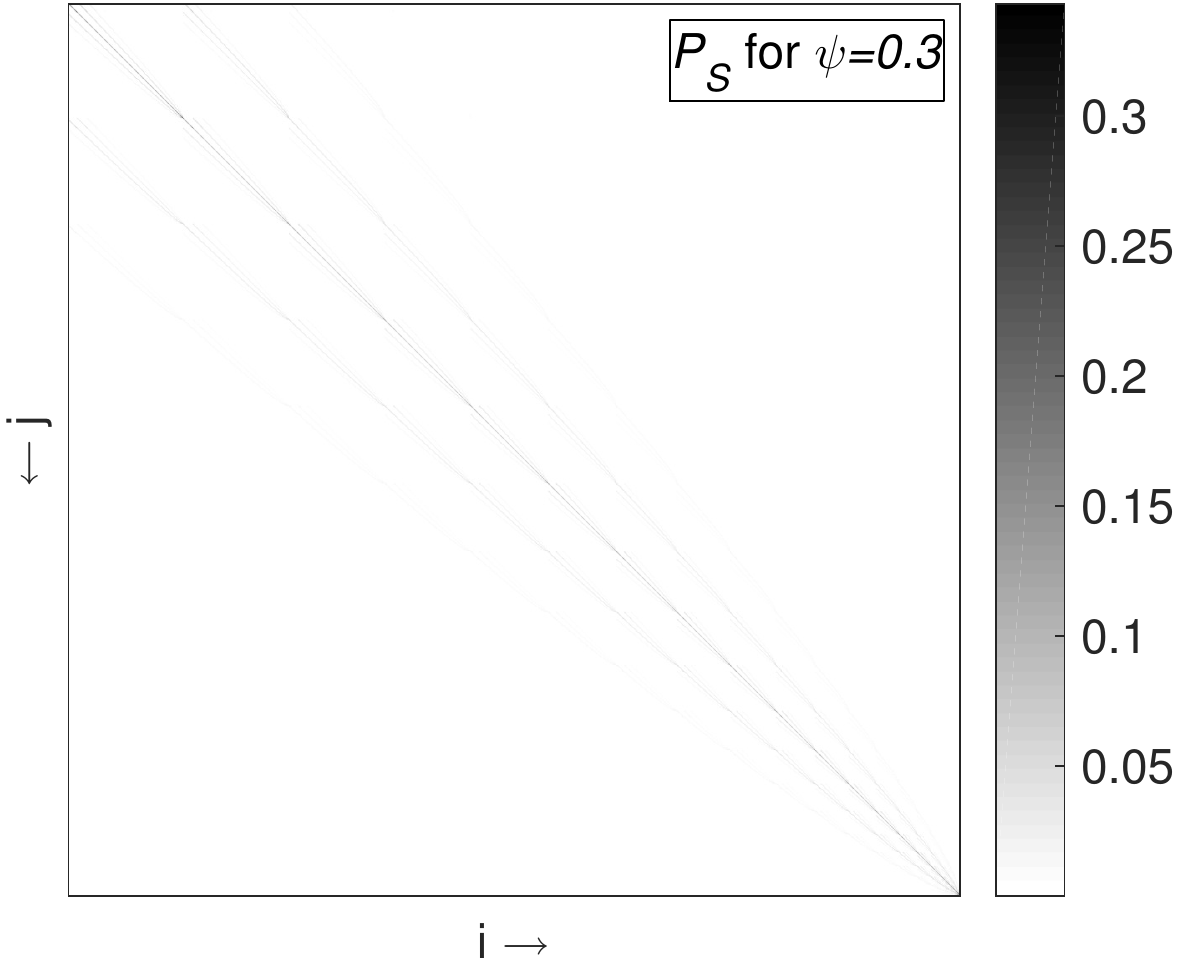}}
	\caption{Transition matrices for the discrete metric with $N=20$ agents and $\card(\Omega)=4$ policies for two different values of $\psi$. The same comments as for Figure~\ref{fig:Pwass} apply.}
	\label{fig:Pdisc}
\end{figure}	

\begin{figure}[t!]
  \includegraphics[clip=true,width=0.49\textwidth]{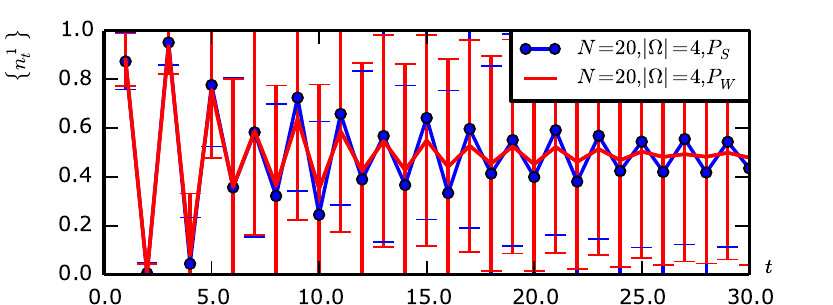}
  \includegraphics[clip=true,width=0.49\textwidth]{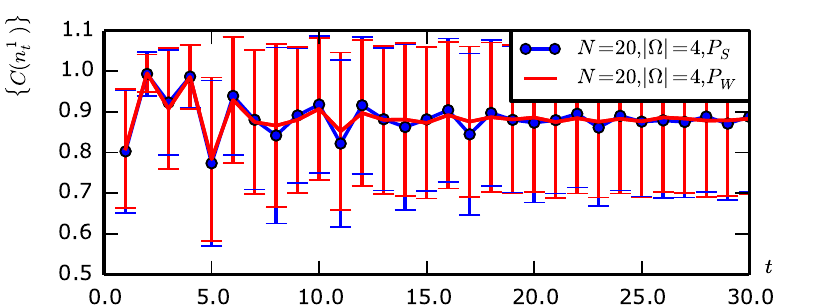} \\
  \includegraphics[clip=true,width=0.49\textwidth]{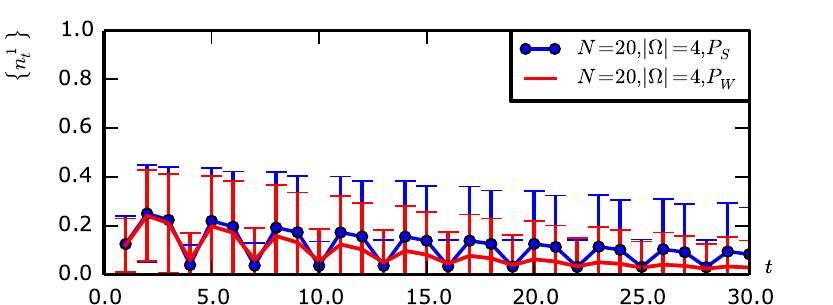}
  \includegraphics[clip=true,width=0.49\textwidth]{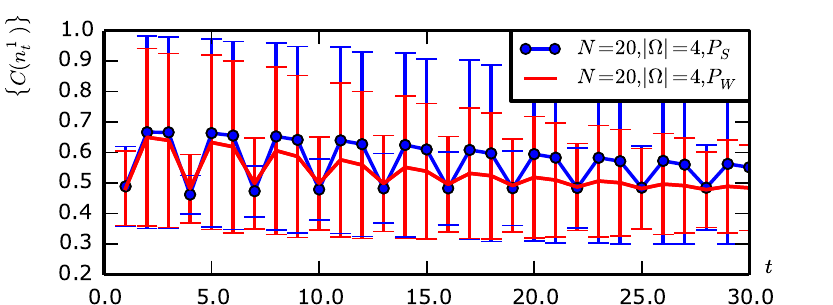} \\
  \includegraphics[clip=true,width=0.49\textwidth]{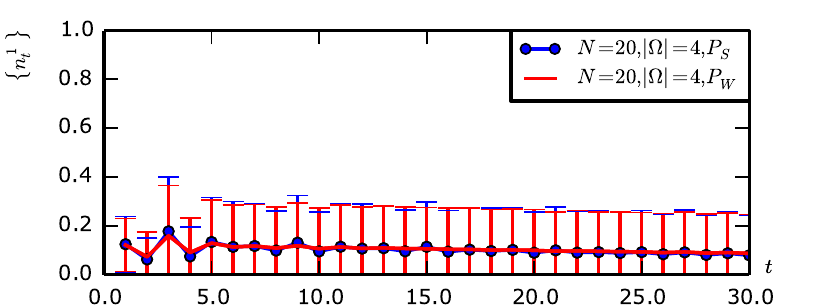}
  \includegraphics[clip=true,width=0.49\textwidth]{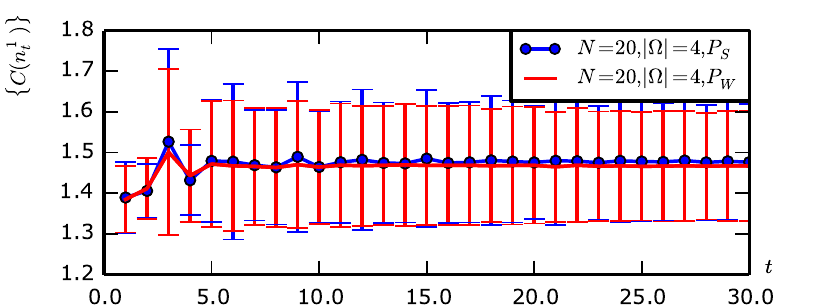}              
  \caption{The behaviour of $N = 20$ agents on the three examples of Figure~\ref{fig:ushaped4settings}, considering both the Wasserstein metric (in red) and the discrete metric (in blue).
    Top left: Evolution of state over time for $x^2 + 0.4$ and $(x^3 + 0.7)/1.7$.
    Top right: The corresponding evolution of the social cost.
    Middle left: Evolution of state over time for $x/10 + 2$ and $1 + 1/(1.1-x)/22$.
    Middle right: The corresponding evolution of the social cost.
    Bottom left: Evolution of state over time for $x + 2$ and $1 + 1/(1.1-x)/22$.
    Bottom right: The corresponding evolution of the social cost.
    Notice the considerable difference in the evolutions of state and cost in the middle and bottom row,
    where $c_2(x)$ are the same and $c_1(x)$ are rather similar.
}
  \label{fig:ushaped4withN20}
\end{figure}


Next, consider an example with $N = 20$ and $\Omega=\{0,1/3,2/3,1\}$, i.e.\ 4 policies and $1771$ distributions.
The $1771 \times 1771$ transition matrices $P_S$ and $P_W$ are displayed in Figures~\ref{fig:Pwass}~and~\ref{fig:Pdisc}.
Note that generating each $P_W$ requires solving $1770\cdot1771/2=1,567,335$ instances of the linear program~\eqref{eq:LP},
as described in Appendix~\ref{sec:app:P}.
The corresponding behaviour is captured in Figure~\ref{fig:ushaped4withN20},
again with each time-series obtained as the mean over 10,000 sample paths
and error bars at one standard deviation.
As in Figure~\ref{fig:ushaped4withN2}, one observes the convergence in distribution throughout.
Unlike in Figure~\ref{fig:ushaped4withN2}, in the middle and bottom rows, one also observes convergence 
to a distribution concentrated around the optima of social cost of approximately 0.387 (middle) and 1.315 (bottom), respectively.
Although the concentration of the limiting distribution around the true optimum cannot be proven generically, its appearance for larger $N$ seems most intriguing.

\section{Related Work}
\label{sec:related}

Our work builds upon a rich history of research in choice modelling \cite{ben1985discrete,Prato2009} and intelligent transportation systems \cite{papageorgiou2007its}.
There, any given user is associated with an origin-destination pair, and hence with a set $\{1, 2, \ldots, M\}$ of routes connecting the two points. 
In the prevailing research direction, it is assumed that the perceived utility of a route $m$ for user $n$ is determined by a scalar $u_t^{n,m} \eqcolon c_m(n_t^m) + \epsilon_t^{n,m}$, where $\epsilon_t^{n,m}$ is a random variable specific to the route, user, and time. The probability that 
user $n$ picks $1 \le m \le M$ at time $t$ is:
  \begin{align}
    \prob_{\theta}&( u_t^{n,m}  < u_t^{n,m'} \; \forall \; m' \in \{1, 2, \ldots, M\} \setminus \{ m \} ) = \label{probabilistic} \\
    \prob_{\theta}&(c_m(n_t^m) + \epsilon_t^{n,m} < c_{m'}(n_t^{m'}) + \epsilon_t^{n,{m'}} \forall \; m' \in \{1, 2, \ldots, M\} \setminus \{ m \} ), \notag
  \end{align}
where $\theta$ is the probability distribution of the vector $\epsilon_t^n \eqcolon (\epsilon_t^{n,1}, \epsilon_t^{n,2}, \ldots,)$.
Clearly, the 
the next-step utilisation can then be expressed 
as an expectation of the perceived utility 
with respect to the distribution
$\theta$,
or as by integrating over $\epsilon_t^n$.
There are a number of popular choices of $\theta$, with varying complexity of the integration over $\epsilon_t^n$.

The first such choice model \cite{ben1985discrete} is the Multinomial Logit (MNL) of \cite{dial1970probabilistic},
which assumes that 
$\epsilon_t^{n,m}$ are independent and identically distributed Gumbel random variables.
There, the equilibrium is computable as a convex program, under the assumption it exists and the cost functions are convex,
but one may struggle to justify the fact that overlapping routes have uncorrelated $\epsilon_t^{n,1}$.
In contrast, in the Multinomial Probit (MNP) model of \cite{daganzo1977stochastic}, $\epsilon_t^n$ is assumed to be a multivariate normal random vector with mean 0 and a given covariance matrix.
There, assuming the conditions for the existence of an equilibrium, one considers unconstrained minimisation of an integral in the computation of the equilibrium values.  
This has been extended in a number of ways, to obtain
nested logit models \cite{ben1985discrete},
path-size logit \cite{hoogendoorn2005path},
multinomial Weibit \cite{castillo2008closed}, 
c-logit \cite{zhou2012c},
and yet more recent models \cite{Mishra2012,ahipasaoglu2014beyond} considering distributionally robust optimisation over all distributions with mean 0 and the given covariance matrix.
These models share a number of features, notably given travel costs $c_m(n_t^m)$ and stochastic $\epsilon_t^{n,m}$,
and the existence of an equilibrium, under certain assumptions.

In an alternative line of research, one assumes that the travel costs $c_n(n_t^m)$ are a stationary random process, or a random variable with a distribution fixed over time. One can obtain a multi-parametric route choice model based on risk measures of the random variable 
and  users' choice based on a level of risk-aversion weighing the risk measures, as in \eqref{piomega}.
Based on early simulations \cite{knoop2008traffic}, a variety of such models have been analysed rigorously, recently, 
 e.g., based on expectation and variance \cite{bell2002risk,nikolova2011stochastic} and 
expectation and value at risk \cite{ordonez2010wardrop,nie2011multi}, which is sometimes known as the percentile equilibrium.
Notice that in this line of research, the travel costs are exogenous random variable,
which again makes it possible to guarantee the existence of an equilibrium, under certain assumptions.

Notice that much of the research on exogenous uncertainty in travel costs,
as well as the research on choice models such as MNL and MNP, 
focuses on the conditions for the existence and computations of equilibria. 
In theory, the restriction to the existence of equilibria is rather unfortunate \cite{haavelmo1974can,kerner2016failure},
 as it obscures the complexity of the possible behaviour of the non-linear dynamics.
In practice, it may be very hard for the transportation authorities to steer the system toward the equilibrium, 
  while maintaining the trust of the users, which may be require the information provided to be truthful, in some sense.
\cite{marevcek2015signaling,marevcek2016signaling} show that when one considers the truthful information provision, 
 announcing a function of $c_m(n_{t-1}^m)$ to all users at time $t$,
 the equilibria do not exists, in general.
For example, in the simple case of truthful provision of costs $c_m(n_{t-1}^m)$, with users' choice based on a scalar perceived utility, 
 it is easy to observe a limit-cycle behaviour in a two-route example, where the traffic alternates between the routes, 
 c.f., Appendix 1 of \cite{marevcek2016signaling}.
Even beyond this very simple control strategy, many standard controllers studied in control theory do fail \cite{Fioravanti2016}
 to stabilise the system.
The study of the related dynamical systems without equilibria \cite{Smith2011} is hence rather more recent,
and often connected to the control of traffic lights (also known as signals),
where it has been long recognised that the non-linear dynamics have limit cycles and yet more complex forms of behaviour.

Still, one may want to consider both the the costs $c_m(n_t^m)$ and their perceptions as non-stationary random processes,
generated by the underlying non-linear dynamics.
Mare\v{c}ek et al.\ \cite{marevcek2015signaling,marevcek2016signaling,MarecekShortenYu2016b} have studied 
measure-theoretic notions of convergence within such models, where 
$c_m(n_t^m)$ are non-stationary random processes, based on a closed-loop model with $u_t^m$ and $v_t^m$ being functions of the history
of $n_{t-1}^m, n_{t-2}^m, \ldots$.
Subsequently, the two-parametric choice model 
$\arg \min_{m=1,\dotsc,M} \omega u_t^m + (1-\omega) v_t^m$ in \eqref{piomega},
 in the spirit of bounded rationality \cite{simon1957administrative}
 is used, as in \cite{bell2002risk,nikolova2011stochastic,ordonez2010wardrop,nie2011multi}.
An overview of the choices of parameters $u_t^m$ and $v_t^m$ is provided in Table~\ref{tab:int-signals}.
In $(\delta, \gamma)$-signalling \cite{marevcek2015signaling}, there is one constant $\delta^m$ per resource $m$ and one 
independent identically distributed (i.i.d.) uniform random variable $\nu^m_t$ with
support $[-\delta^m/2,\delta^m/2]$ per resource $m$ and time step $t$.
Subsequently, one defines $u_t^m$ and $v_t^m$ as $c_m(n^m_{t-1}) + \nu_t^m - \delta^m/2$ and $c_m(n^m_{t-1}) + \nu_t^m + \delta^m/2,$ respectively, 
resembling \eqref{probabilistic}.
In $r$-extreme signalling \cite{marevcek2016signaling}, 
$(u_t^m, v_t^m)$ is the minimum and maximum over a time window of size $r$, i.e.,
within $n_{t-1}^m, n_{t-2}^m, \ldots, n_{t-r}^m$.
Finally, in $r$-supported signalling \cite{MarecekShortenYu2016b}, one minimises social cost $C$ over the sub-intervals of minimum and maximum over a time window of size $r$,
with certain restrictions on $(u_t^m, v_t^m)$ captured by the inclusion in a set $P(S_t, \Omega)$, so as to minimise the next-step social cost.
Many options remain to be explored.

\section{Conclusions and Future Work}

We have studied the stability of resource allocation, where a central authority broadcasts two values for each resource and agents take a convex combination of the two values, with the
distribution of the convexifying coefficients across the population varying over time
in a non-stationary fashion.
Our contributions are as follows:
\begin{itemize}
\item A behavioural model, which does not simplify the nonlinear dynamics to a fixed point and considers evolution of the levels of risk-aversion in the population governed by a Markov chain.
\item Novel means of information provision in resource-allocation problems, based on exponential smoothing of the past costs and past variance in the costs.
\item A convergence result for the particular behavioural model and the particular means of information provision, whose proof technique may be more widely applicable.
\end{itemize}
This may spur much further research, along several directions.

One direction may focus on the wide variety of means of deriving the information to broadcast, 
some of which are suggested in Table~\ref{tab:int-signals}.
Outside of (A) further variants of the exponential smoothing, one could consider 
(B) the provision of the expected value and variance, c.f. \cite{bell2002risk,nikolova2011stochastic},
(C) the expectation and value at risk for a given coefficient $\alpha$ and distribution function $L$ with support $\{ c_m(n_t^m) \}_t$,
 c.f. \cite{ordonez2010wardrop,nie2011multi},
and (D) the expectation and conditional value at risk, also known as the expected shortfall, with the same distribution.
For both the previously studied variants and those suggested in Table~\ref{tab:int-signals}
for the first time, bounds on the concentration of the limiting distribution around
its mean, and conditions for the mean being socially optimal, remain a major open problem.
One may also bound the rate of convergence, c.f. \cite{slkeczka2011rate},
 using the coupling argument of Hairer \cite{hairer2005coupling}.
Although such arguments are somewhat technical, 
 they allow for a principled use of simulations, with confidence bounds for a particular simulation set-up.

We also believe that the model of behaviour we have introduced is rich enough so as to allow
the application in many other domains.
For example in power systems, prices may be adjusted dynamically, but are currently not announced in real time, so as to prevent additional fluctuations in demand, leading to additional costs.
A key open problem hence is how to provide the information on prices in a predictable and socially optimal fashion. 
One option may be to provide retail consumers with information on both the costs of energy and levels of pollution.
Assuming each customer weighs these differently, one may apply a similar model for the evolution of the distribution of such 
a coefficient across the population.
Similar reasoning may be applicable more broadly.

\paragraph{Acknowledgement} Funding from the European Union Horizon 2020 Programme (Horizon2020/2014-2020) under grant agreement number 688380 is gratefully acknowledged.

\begin{landscape}
\begin{table}[h]
    \centering
Related work:\\[2mm]
    \begin{tabularx}{1.55\textwidth}{| l | p{4.5cm} p{5.5cm}  X |}
        \hline
Ref.                         & Name                          & $u_t^m$    & $v_t^m$ \\ \hline
\cite{marevcek2015signaling}  & $(\delta, \gamma)$-signalling & $c_m(n^m_{t-1}) + \nu_t^m - \delta^m /2$ & $c_m(n^m_{t-1}) + \nu_t + \delta^m /2$ \\ 
\cite{marevcek2016signaling}  & $r$-extreme signalling        & $\arg \min_{j=t-r,\ldots,t-1} \{ c_m(n_j^m) \}$   &  $\arg \max_{j=t-r,\ldots,t-1} \{ c_m(n_j^m) \} $ \\ 
\cite{MarecekShortenYu2016b} & $r$-supported signalling      & $\proj_{u_t^m} \arg \min_{ (u_t^m, v_t^m) \in P(S_t, \Omega)} C(n_t)$ & $\proj_{v_t^m} \arg \min_{ (u_t^m, v_t^m) \in P(S_t, \Omega)} C(n_t)$ \\ 
Here                         & Exponential smoothing         & $q_1 u_{t-1}^m + (1-q_1)c_m(n_{t-1}^m)$ & $q_2 v_{t-1}^m + (1-q_2)\,\bigl|c_m(n_{t-1}^m)-u_{t-1}^m\bigl|$       \\ \hline
    \end{tabularx}\\[6mm]
Future work suggestions:\\[2mm]
    \begin{tabularx}{1.55\textwidth}{| l | p{4.5cm} p{5.5cm}  X |}
        \hline
Ref.                         & Name                          & $u_t^m$    & $v_t^m$ \\ \hline
A  & Variants of smoothing         & $q_1 u_{t-1}^m + (1-q_1)c_m(n_{t-1}^m)$   &  $q_2 u_{t-1}^m + (1-q_2)c_m(n_{t-1}^m)$  \\ 
B  & Mean and variance             & $\frac{1}{r} \sum_{j=t-r,\ldots,t-1} c_m(n_j^m)$   &  $\frac{1}{r} \sum_{j=t-r,\ldots,t-1} (c_m(n_j^m) - u_t^m)^2$  \\ 
C  & Mean and value at risk        & $\frac{1}{r} \sum_{j=t-r,\ldots,t-1} c_m(n_j^m)$   &  VaR$_\alpha \coloneq \inf\{l \in \mathbb{R}: \prob(L>l)\le 1-\alpha\}$  \\ 
D  & Mean and conditional VaR            & $\frac{1}{r} \sum_{j=t-r,\ldots,t-1} c_m(n_j^m)$   &  CVaR$_\alpha \coloneq \frac{1}{\alpha}\int_0^{\alpha} \mbox{VaR}_{\gamma} d\gamma $ \\ \hline
    \end{tabularx}\\[6mm]
\caption{A comprehensive overview of the related work and suggestions for future work within two-parameter information provision and two-parameter route choice formulations.}
\label{tab:int-signals}
\end{table}
\end{landscape}

\bibliographystyle{spmpsci}
\bibliography{behavioural,traffic,experiments,bandits,ifs}

\newpage
\appendix
\section{The Transition Matrix $P$}\label{sec:app:P}
In order to run simulations, and possibly also to obtain tighter theoretical results than Theorem~\ref{thm:main}, based on the model of population dynamics proposed in Section~\ref{ssec:popdyn}, it is necessary to compute a transition matrix $P$ for the Markov chain generating the driver populations. We elaborate here on this procedure.


\subsection{$P$ to Encode Time Dependence}\label{ssec:app:timedep}

If $P$ is used to model the dependence of the populations on an underlying temporal process, such as the time of day changing, the Markov chain would resemble a cycle: the state is forced to progress from morning to noon to afternoon and so on, until the day is over and everything repeats, there would be no transitions backwards. As explained in Section~\ref{ssec:popdyn}, the simplest encoding of that would yield a matrix $P$ that is a unit shift, with ones on the first superdiagonal and a 1 one in the left bottom corner (if the different times of day are numbered appropriately). This would reflect a deterministic behaviour and thus again not be very realistic. Instead, there should be several possible populations per time of day. To illustrate the idea, consider the same example as in section~\ref{ssec:popdyn}, but assume there are now two different possible noon populations that both have the same probability. Then, $P$ would be
\[	
		P=\bmat{0 & .5 & .5 & 0 & 0 \\
				0 & 0 & 0 & 1 & 0  \\
				0 & 0 & 0 & 1 & 0  \\
				0 & 0 & 0 & 0 & 1 \\
				1 & 0 & 0 & 0 & 0}	,
\]
reflecting that if it was ``morning'' at time $t$, at time $t+1$ there is a 0.5 probability for each of the noon populations, but that in either case at $t+2$ we will have the afternoon population. 

Formally: Assume there are $T$ times of day, each with a set of $N_\ell$ different populations, all of which are equally probable. That would be reflected in a transition matrix
\[
	P = \left[\begin{array}{c;{1pt/1pt}c;{1pt/1pt}c;{1pt/1pt}c;{1pt/1pt}c}
			\vphantom{\ddots}0_{N_1 \times N_1} & \frac{1}{N_2}\one_{N_1 \times N_2 } & 0_{N_1 \times N_3} &\cdots&  0_{N_1 \times N_T}\\ \cdashline{1-5}[1pt/1pt] 
			0_{N_2 \times N_1} & 0_{N_2 \times N_2} & \frac{1}{N_3}\one_{N_2 \times N_3 } & \ddots & \vdots \\ \cdashline{1-5}[1pt/1pt] 
			\vdots & & \ddots & \ddots & \\ \cdashline{1-5}[1pt/1pt] 
			0_{N_{T-1} \times N_1} & &  &\ddots &\frac{1}{N_T}\one_{N_{T-1} \times N_T } \\ \cdashline{1-5}[1pt/1pt]
			\vphantom{\ddots}\frac{1}{N_1}\one_{N_T \times N_1 } & 	0_{N_T \times N_2} & \cdots &  \hphantom{0_{N \times N_1}} &	0_{N_T \times N_T}		
		 \end{array}
		 \right],
\]
where $0_{M \times N}$ and $\one_{M \times N }$ denote $M \times N$ matrices of all zeros and ones, respectively.

\subsection{Earth Mover's Distances (EMDs)}\label{ssec:app:EMD}
The Earth Mover's Distance (EMD) between histograms (which in the text have been also called populations and distributions) gets its name from interpreting a histogram as a pile of dirt. To transform one histogram into another, the dirt has to be rearranged; the \textit{ground distance} between two to bins (or categories) is the work necessary to move one unit of dirt from one bin to the other, and the distance $\Delta(\eta, \gamma)$ between two histograms $\eta$ and $\gamma$ (with the same categories) is then defined as the minimal work necessary to transform one histogram into the other. For our purposes, the dirt corresponds to the agents, and the bins correspond to the policies $\Omega$.

The choice of ground distance $h(\omega,\omega')$ between policies $\omega$ and $\omega'$ now depends on  what brings about the change of individual policies. As already mentioned in Section~\ref{ssec:popdyn}, if an individual driver's policy does not change, but the driver instead leaves the road network and is replaced by a new driver, then $h(\omega,\omega')$ should only reflect that an agent has to be replaced, but be independent of $\omega$ and $\omega'$:
\[
	h(\omega,\omega') = 1-\delta_{\omega,\omega'} =  \begin{cases}
	0 & \text{if } \omega=\omega'\\
	1 & \text{if } \omega\neq\omega'.
	\end{cases}
\]

If instead the drift of an individual driver's policy is to be modelled, $h(\cdot,\cdot)$ should reflect the fact that drastic changes are less likely to occur than slight ones; in other words the ground distance $ h(\omega,\omega')$ needs to be greater if the difference $\abs{\omega-\omega'}$ is. One particularly simple choice is
\[
	h(\omega,\omega') = \abs{\omega-\omega'}.
\]
This special case of an EMD is better known -- especially in the case of continuous distributions -- as the Wasserstein distance, and we will call it by this name, too.

In either case, computation of $\Delta(\cdot,\cdot)$ requires solving an optimization problem, as there are many ways to ``rearrange the dirt,'' but not all of them are optimal. For any choice of $h(\cdot,\cdot)$, the EMD can be computed as the solution of a linear program (LP), and if, as in our case, the items counted by the histograms cannot be divided into smaller units, an integer linear program (ILP):
Let $x_{ij}$ denote the amount of agents changing policy from $\omega_i$ to $\omega_j$, and let the $\card(\Omega)\times \card(\Omega)$ matrix $H$ have elements $H_{ij}=h(\omega_i,\omega_j)$. Then the minimal amount of work necessary to transform $\eta$ into $\gamma$ is the optimal cost of
\begin{equation}\label{eq:LP}\tag{LP}
\begin{aligned}
\text{minimize } &\qquad& & y  = \trace(H^T X) = \sum_{i=1}^M\sum_{j=1}^M x_{ij} h(\omega_i,\omega_j) \\
\text{subject to } &&& \left\{ \begin{aligned}
x_{ij} &\geq0 & x_{ij}&\in\integers\\
\one^T X &= \eta^T \\
X \one &= \gamma.
\end{aligned}\right.
\end{aligned}
\end{equation}
Then,
\[
	\Delta(\eta, \gamma) \coloneq y^\text{opt}.
\]

\begin{remark}
	For the substitution metric, we have \[H=\one_{\card(\Omega)\times \card(\Omega)} - I_{\card(\Omega)}.\] 
\end{remark}

\begin{remark}
	\eqref{eq:LP} has the form of a so-called transportation problem. Hence, if the problem data is integer, then there is always an integer optimum. In the present case, $\eta$ and $\gamma$ are integer by definition, so the only concern is $H$ for the Wasserstein metric. However, if we choose all $\omega_i$ to be rational, then multiplying $H$ by their least common denominator (or simply all denominators) renders $H$ integer without changing the optimal solutions $X^\text{opt}$. This class of problems also allows more efficient solution than a general LP via special algorithms, see e.g.~\cite{luen_opt}.
\end{remark}

The probability of transitioning from $\eta$ to $\gamma$ should of course be smaller if the distance $\Delta(\eta,\gamma)$ (i.e.\ the ``work'' required) is greater; hence, if now $\vec{\eta}_i$ and $\vec{\eta}_j$ denote two populations, then $p_{ij}$ should be a bounded, nonnegative, decreasing function of $\Delta(\vec{\eta}_i,\vec{\eta}_j)$. One class of functions satisfying all requirements is $p_{ij} = \alpha \psi^{\Delta(\vec{\eta}_i,\vec{\eta}_j)}$, where $\psi\in(0,1)$ and $\alpha$ is a normalization factor to make the resulting matrix $P$ row-stochastic. Then, we can construct $P$ from
\begin{align*}
	\tilde{p}_{ij} &= \psi^{\Delta(\vec{\eta}_i,\vec{\eta}_j)}\\
	{p}_{i j} &= \tilde{p}_{ij}\biggm/ 
	\raisebox{-1ex}{$\displaystyle\left( \sum_k \tilde{p}_{i k} \right)$} .
\end{align*}
The parameter $\psi$ can be interpreted as a measure of the probability of one agent changing policy (for whatever reason). For $\psi=1$, we get $p_{ij} = 1/K$~$\forall i,j$, i.e.\ we recover the i.i.d.\ case of e.g.~\cite{marevcek2016signaling}.

\end{document}